\documentclass[11pt]{amsart}
\usepackage{amsmath}
\usepackage{amsthm,amssymb}
\usepackage{hyperref}
\usepackage{amssymb}
\usepackage{graphicx}
\usepackage[utf8]{inputenc}
\input xy
\xyoption{all}
\usepackage{color}
\usepackage{enumerate}
\newtheorem{theorem}{Theorem}[section]

\newtheorem{lemma}[theorem]{Lemma}

\newtheorem{corollary}[theorem]{Corollary}
\newtheorem{proposition}[theorem]{Proposition}
\theoremstyle{definition}

\newtheorem{remark}[theorem]{Remark}
\newtheorem{problem}[theorem]{Problem}

\newcommand{\diag}{\mbox{\rm diag}}

\newcommand{\ch}{\mbox{\rm char}}

\begin{document}

\title[Products of unipotent elements ...]{Products of unipotent elements in \\ certain algebras}

\author[M. H. Bien]{M. H. Bien$^{1,2}$}
\author[P. V. Danchev]{P. V. Danchev$^{3}$}
\author[M. Ramezan-Nassab]{M. Ramezan-Nassab$^{4}$}
\author[T. N. Son]{T. N. Son$^{1,2}$}
	
\address{[1] Faculty of Mathematics and Computer Science, University of Science, Ho Chi Minh City, Vietnam}
	
\address{[2] Vietnam National University, Ho Chi Minh City, Vietnam}
	
\address{[3] Institute of Mathematics and Informatics, Bulgarian Academy of Sciences, 1113 Sofia, Bulgaria}
	
\address{[4] Department of Mathematics, Kharazmi University, 50 Taleghani Street, Tehran, Iran}

\medskip
	
\email{M. H. Bien: mbbien@hcmus.edu.vn; \newline
		Peter V. Danchev: danchev@math.bas.bg; pvdanchev@yahoo.com\newline
		M. Ramezan-Nassab: ramezann@khu.ac.ir \newline
		T. N. Son: trannamson1999@gmail.com}

	
\begin{abstract}
Let $F$ be a field with at least three elements and $G$ a locally finite group. This paper aims to show that if either $F$ is algebraically closed or the characteristic of $F$ is positive, then an element $\alpha $ in the group algebra $FG$ is a product of unipotent elements if, and only if, $\alpha$ lies in the first derived subgroup of the unit group $(FG)^*$ of $FG$. In addition, $\alpha$ is a product of at most three unipotent elements.

Moreover, we explore some crucial properties satisfied by certain algebras like the connection between unipotent elements of index $2$ and commutators as well as we investigate the unipotent radical of a group algebra by showing that the group algebra of a finite group over an infinite field cannot have a unipotent maximal subgroup. In particular, we apply these results to twisted group algebras. 
\end{abstract}

\maketitle


\section{Introduction}
	
Let $R$ be an algebra. An element $x\in R$ is called \textit{unipotent} if $1-x$ is nilpotent. Additionally, if $k$ is the smallest positive integer such that $(1-x)^k=0$, then $x$ is called \textit{unipotent of index $k$}. Unipotent elements are very important in the study of the algebraic structure of algebras. One of the nicest results concerning unipotent elements is that we may describe its inverse explicitly. In fact, if $x\in R$ is unipotent of index $k$, then $x$ is invertible and $$x^{-1}=(1-(1-x))^{-1}=1+(1-x)+\cdots+(1-x)^{k-1}.$$ In some papers concerning to Hartley's conjecture for group algebras with units satisfying group identities, many authors use unipotent elements of index $2$ (e.g., see \cite{Pa_GiMiSe_10, Pa_GiSeVa_00, Pa_Li_99, Pa_RaBiAk_21}); in fact, to find elements which generate a free subgroup, many authors seek unipotent elements (e.g., see \cite{Pa_GoPa_04, Pa_GoMa_01}, for a survey on this topic, we refer the interested readers to \cite{Pa_GoAn_13}). Notice that, in \cite{Pa_MaSe_97}, it is questioned that in a torsion-free ring $R$ with two elements $a,b\in R$ such that $a^2=b^2=0$, when does $1 + a$ and $1 + b$ generate a free subgroup?
	
Throughout the present article, we denote by $R^*$ the unit group of $R$ and by $R'$ the first derived subgroup of $R^*$, i.e., the subgroup generated by all the commutators $[x,y]=xyx^{-1}y^{-1}$, where $x,y\in R^*$. It seems to be an interesting topic to investigate the subgroup of $R^*$ generated by all unipotent elements in $R$. One special case which has received much attention is special linear groups. Let $F$ be a field and $n$ a positive integer. It is well known that every unipotent matrix in $\mathrm{M}_n(F)$ is similar to an {\it unitriangular} matrix, i.e., a (upper) triangular matrix whose all diagonal entries are $1$ (see, e.g., \cite[Lemma~3.2]{Abdi}), which implies that every unipotent matrix belongs to the special linear group $\mathrm{SL}_n(F)$. Conversely, observe that each matrix in $\mathrm{SL}_n(F)$ is a product of elementary matrices, which are unipotent of index $2$, so  $\mathrm{SL}_n(F)$ is generated by all unipotent matrices in $\mathrm{M}_n(F)$. In 1986, Fong and Sourour showed that if $F=\mathbb{C}$, the field of complex numbers, then every matrix in $\mathrm{SL}_n(\mathbb{C})$ is a product of at most three unipotent matrices (see \cite{Pa_FoSo_86}). In the case of index $2$, it was shown that every matrix in $\mathrm{SL}_n(\mathbb{C})$ is a product of at most four unipotent matrices of index $2$ (see \cite{Pa_WaWu_91}).
	
Let $G$ be a group and $F$ a field, and suppose $FG$ is the group algebra of $G$ over $F$. In this paper, we consider nilpotent elements and products of nilpotent elements in group algebras. As mentioned above, in group algebras, unipotent elements play a key role in the study of the units of group algebras (e.g, see \cite{Pa_ArHa_98,Pa_Bov_95,Pa_DaMa_18,Pa_GoVe_09,Pa_Ha_07,Pa_Ku_19}). Especially, Danchev and Al-Mallah investigate groups $G$ and fields $F$ such that all units of $FG$ are unipotent elements \cite{Pa_DaMa_18}.  	
	
The first aim, which motivates writing of this paper, is devoted to the investigation of the subgroup of the unit group $(FG)^*$ generated by all unipotent elements in $FG$ in the case where $F$ is algebraically closed and $G$ is locally finite. Our second motivating goal is concerned with the behavior of unipotent elements of index $2$ situated with commutators by calculating in some critical situations the upper bound of the number of unipotents. Finally, in order to synchronize the pursued objectives, our third purpose deals with the unipotent radical of some special subgroups in (twisted) group algebras.

Concretely, our work is systematically organized as follows: In Section~2, we will prove that if $G$ is locally finite and $F$ is a field, then $(FG)'$  is generated by all unipotent elements. In the case when $F$ is either algebraically closed or of positive characteristic, then an element in $(FG)'$ is a product of at most three unipotent elements (see, for instance, Theorems~\ref{twist1} and \ref{twist2}). This may be considered as an analog of the well-known result for matrix algebras over fields (see, e.g., \cite{Pa_BiDuHaSo_22}). Furthermore, in the next two sections, namely Sections 3 and 4, we consider some special cases when the group algebra $FG$ is semi-simple and our chief result of this case is to prove that every element in $(FG)'$ is a product of at most two commutators of unipotent elements of index $2$ (see, for example, Theorems~\ref{final} and \ref{bound}). In the final fifth section, we study the unipotent radical of two important classes of the general skew linear groups. The most important achievements are structured in Theorems~\ref{subnormal} and \ref{234}.

	
\section{The derived subgroup of some algebras and unipotent elements}

Let $D$ be a division ring and $n\geq 2$ an integer. Throughout the text, ${\rm SL}_n(D)$ denotes the kernel of the classical Dieudonn$\acute{\rm e}$ determinant. It is well known that if $D\neq {\mathbb F}_2$, then $(\mathrm{GL}_n(D))' =\mathrm{SL}_n(D)$. For $n=1$, we agree that $\mathrm{SL}_1(D)=D'$. 

In this aspect, we first have the following plain but useful claim.
	
\begin{lemma}\label{l3.1}
Let $D$ be a division ring with at least three elements and $n\geq 2$ an integer. Then $A\in \mathrm{M}_n(D)$ is a product of unipotent matrices if and only if $A\in \mathrm{SL}_n(D)$.
\end{lemma}

\begin{proof}
This is directly inferred from  \cite[Lemma~3.2]{Abdi} and \cite[Theorem~3, p.~137]{Bo_Dra_83}.
\end{proof}


Note that the matrix

$$A=\left(\begin{matrix}
		1&1\\0&1
		\end{matrix}\right)\in \mathrm{GL}_2({\mathbb F}_2) =\mathrm{SL}_2({\mathbb F}_2)$$
is a unipotent element, but $A\notin  (\mathrm{GL}_2(D))'$ (see, e.g., \cite[Proposition~2]{Pa_SoDuHaBi_22} and its proof). Thus, in Lemma~\ref{l3.1} (and almost throughout the paper), we shall assume that $|F|>2$.

\medskip

Let $R, R_1,\dots,R_t$ be algebras and $x\in R$. If $x$ is unipotent, then it is readily checked that so are both $x^{-1}$ and $yxy^{-1}$ for every $y\in R^*$. If, however, each $x_i\in R_i$ is a product of at most $k$ unipotent elements, then $x=(x_1,\ldots,x_t)\in R_1\times  \cdots \times R_t$ is a product of at most $k$ unipotent elements in  $R_1\times \cdots \times R_t$. These facts are used frequently in the sequel without concrete referring.

If $R$ is a semi-simple ring, then by the Artin-Wedderburn theorem we can write $R\cong\prod_{i=1}^t\mathrm{M}_{n_i}(D_i)$, where all $n_i$'s are positive integers and all $D_i$'s are division rings. By saying ``each Wedderburn component of $R$ is not a non-commutative division ring", our mean is that in the above decomposition, for each $i$, either $n_i>1$, or $n_i=1$ and $D_i$ is a field.

We thus arrive at the following useful statement.

\begin{proposition} Suppose that $F$ is a field with at least three elements and  $R$ is an $F$-algebra.
				\begin{itemize}
					\item[(1)] If $R$ is left Artinian, then every element which is a product of unipotent elements in $R$ belongs to $R'$.
					\item[(2)] If $R$ is semi-simple such that each Wedderburn component of $R$ is not a non-commutative division ring, then $\alpha\in R'$ if and only if $\alpha$ is a product of unipotent elements.
				\end{itemize}
\end{proposition}

\begin{proof} (1) Since $R$ is  a left Artinian ring, one finds that $$R/J(R)\cong \mathrm{M}_{n_1}(D_1)\times  \dots \times \mathrm{M}_{n_t}(D_t),$$
where $J(R)$ denotes the Jacobson radical of $R$, $n_i$'s are positive integers and $D_i$'s are division rings.  Thus,
$$R^*/ (1+J(R))\cong (R/J(R))^*\cong \mathrm{GL}_{n_1}(D_1)\times \dots \times \mathrm{GL}_{n_t}(D_t).$$
				
If $\alpha\in R$ is a product of unipotent elements, then so is $\overline \alpha$ in $R/J(R)$, which implies by Lemma~\ref{l3.1} that
$$\overline \alpha\in \mathrm{SL}_{n_1}(D_1)\times \dots \times \mathrm{SL}_{n_t}(D_t)\cong (R^*/ (1+J(R))'\cong R'/(R'\cap (1+J(R)).$$
Therefore,  $\alpha\in R'$.
				
(2) Assume that $R$ is semi-simple and each Wedderburn component of $R$ is not a non-commutative  division ring, that is, in the Wedderburn decomposition above, for every $i$, either $n_i>1$ or $n_i=1$ and $D_i$ is a field.  According to the first part, we need to show if $\alpha\in R'$, then $\alpha$ is a product of unipotent elements. Assume that $\alpha=(a_1,a_2,\ldots,a_t)$. For every $1\le i\le t$, $a_i\in \mathrm{SL}_{n_i}(D_i)$. By Lemma~\ref{l3.1}, $a_i$ is a product of unipotent matrices in $\mathrm{M}_{n_i}(D_i)$. Hence, $\alpha$  is a product of unipotent elements in $R^*$.
\end{proof}


The next comments are helpful in order to understand more deeply the difficulty in the situation of rational group algebras.
		
\begin{remark} Note that if $K_8$ is the quaternion group, then the rational group algebra $\mathbb{Q}K_8$ contains a non-commutative division ring in its Wedderburn components. In fact, $\mathbb{Q}K_8\simeq\mathbb{Q}  \times\mathbb{Q}  \times \mathbb{Q} \times \mathbb{Q} \times \mathbb{H}_\mathbb{Q}$. Clearly, $(\mathbb{Q}K_8)'$ contains strictly the subgroup of $(\mathbb{Q}K_8)^*$ generated by unipotent elements.

Nevertheless, one approach that might avoid this difficulty with division rings in some specific cases, is the following: embed the finite group $G$ in the symmetric group $S$ on its elements, i.e., the symmetric group generated by the elements of $G$. Thus, the rational group algebra $\mathbb{Q}G$ obviously embeds in $\mathbb{Q}S$. Now, $\mathbb{Q}S$ is a direct sum of full matrix rings over $\mathbb{Q}$, that is, its Wedderburn components are all full matrix rings over $\mathbb{Q}$. Furthermore, start with an element of $\mathbb{Q}G$ and work in $\mathbb{Q}S$, not trying to work in $\mathbb{Q}G$. For example, we know that $S$ has trivial center. So, if we need a non-central group element, choose it from $G$ (assuming $G$ is not $1$). Also, if we need two non-commuting group elements, choose them from $G$ (assuming $G$ is non-abelian), etc.

To demonstrate that this approach does not work in general, let us embed the aforementioned group $K_8$ in the symmetric group $S_8$. Then, as we noted, $\mathbb{Q}K_8$ embeds in $\mathbb{Q}S_8$, which is a finite direct sum of full matrix ring over $\mathbb{Q}$. Thus, each element of the derived group of $\mathbb{Q}K_8$ as an element of $\mathbb{Q}S_8$ is a product of at most three unipotents by Lemma~\ref{l3.6}, while the only unipotent of $\mathbb{Q}K_8$ is $1$, that is nonsense.

However, this idea may work successfully for the matrix ring ${\rm M}_n(\mathbb{Q}G)$, but unfortunately we do not have a favorable solution yet.
\end{remark}

Now, we find the smallest number $m$ such that an element in the subgroup of $(F^\tau G)^*$ generated by all unipotent elements can be written as a product of at most $m$ unipotent elements. In fact, in the case when  $G$ is locally finite and $F$ is algebraically closed or $\ch F>0$, we will show in what follows that exactly $m\le 3$.

The following two lemmas are pivotal.

\begin{lemma} \cite[Corollary]{Pa_So_86}\label{l3.6}
Let $F$ be a field and $n$ a positive integer. Any $A\in  \mathrm{SL}_n(F)$ is a product of at most three unipotent elements, and if either $A$ is non-central or $A=\mathrm{I}_n$, then $A$ is a product of at most two unipotent matrices.
\end{lemma}

	
\begin{lemma}\label{product-unipotent}
Let $R$ be a ring whose Jacobson radical $J(R)$ is nilpotent. If $x\in R^*$ is unipotent and $y\in J(R)$, then $(1+y)x$ is a unipotent element.
\end{lemma}

\begin{proof}
Consider the element $1-(1+y)x=(1-x) -yx$. Then, $\alpha:=1-x$ is nilpotent and $\beta:=-yx\in J(R)$. Hence, there exist two positive integers $m_1$ and $m_2$ such that $\alpha^{m_1}=0$ and $J(R)^{m_2}=0$. Put $m=m_1m_2$. Then $(\alpha+\beta)^m=\sum z_{i_1}z_{i_2}\dots z_{i_m}$ where $z_{i_j}\in \{\alpha,\beta\}$. If $\beta$ occurs $m_2$ times in the product $z_{i_1}z_{i_2}\dots z_{i_m}$, then $z_{i_1}z_{i_2}\dots z_{i_m}\in J(FG)^{m_2}=0$; otherwise, $\alpha^{m_1}$ occurs in $z_{i_1}z_{i_2}\dots z_{i_m}$, then $z_{i_1}z_{i_2}\dots z_{i_m}$ is also equal to $0$. Thus, $(\alpha+\beta)^m=0$, which completes the proof.
\end{proof}

	
We are ready to show the first main result of this paper.

\begin{theorem}\label{twist1}
Let $F^\tau G$ be a twisted group algebra of a locally finite group $G$ over an algebraically closed field $F$.
Then, $\alpha\in (F^\tau G)'$ if and only if $\alpha$ is a product of at most three unipotent elements.
\end{theorem}

\begin{proof}
Assume that $\alpha=[\alpha_1,\beta_1]^{t_1}\cdots[\alpha_k,\beta_k]^{t_k}$, where $\alpha_i,\beta_i\in (F^\tau G)^*$ and $t_i\in\mathbb{Z}$ for $1\leq i\leq k$. Let $H$ be a subgroup of $G$ generated by the supports of all the $\alpha_i,\beta_i,\alpha_i^{-1}$ and  $\beta_i^{-1}$. Then $H$ is a finite group and  $\alpha\in (F^\tau  H)'$. Therefore, it suffices to show the case when $G$ is a finite group.

Now, $F^\tau G/J(F^\tau G)\cong \mathrm{M}_{n_1}(D_1)\times \dots \times \mathrm{M}_{n_t}(D_t)$,
where $n_i$'s are positive integers and $D_i$'s are division rings which are finite dimensional over $F$.  Moreover, since $F$ is algebraically closed, one has $D_i=F$.  Thus,
$$(F^\tau G)^*/ (1+J(F^\tau G))\cong (F^\tau G/J(F^\tau G))^*\cong \mathrm{GL}_{n_1}(F)\times \dots \times \mathrm{GL}_{n_t}(F).$$
		
Assume that $\alpha\in (F^\tau G)'$. Indeed, one has that
$$\overline \alpha =(a_1,a_2,\dots,a_t)\in \mathrm{GL}_{n_1}(F)\times \dots \times \mathrm{GL}_{n_t}(F).$$
For every $1\le i\le t$, by hypothesis that $\alpha\in (F^\tau G)'$, one has $a_i\in \mathrm{SL}_{n_i}(F)$. If $a_i\ne \mathrm{I}_{n_i}$ and $a_i$ is central in $\mathrm{SL}_{n_i}(F)$, then put $b_i=\begin{pmatrix}1&1\\0&1\end{pmatrix}\oplus \mathrm{I}_{n_i-2}$; otherwise, $b_i=\mathrm{I}_{n_i}$. Then, $a_i b_i\in \mathrm{SL}_{n_i}(F)$ and either $a_ib_i$ is non-central or $a_ib_i=\mathrm{I}_{n_i}$, which implies by Lemma~\ref{l3.6}, $a_i b_i$ is a product of at most two unipotent matrices in $\mathrm{SL}_{n_i}(F)$. Hence, putting $\overline\beta = (b_1,b_2,\dots,b_t)$, we have $\overline {\alpha\beta}=\overline \gamma_1 \overline \gamma_2$, where $ \overline \gamma_1$ and $\ \overline \gamma_2$ are two unipotent elements  in $(F^\tau G)^*/ (1+J(F^\tau G))$. It implies that $\alpha\beta =\gamma_1\gamma_2(1+\delta)$ for some $\delta\in J(F^\tau G)$.
Note that  as $J(F^\tau G)$ is nilpotent, $\gamma_1$ and $\gamma_2$ are unipotent elements in $(F^\tau G)^*$.
Now $$\alpha=\gamma_1\gamma_2(1+\delta)\beta^{-1}.$$
Since $\beta^{-1}$ is unipotent, Lemma~\ref{product-unipotent} implies that $\alpha$ is a product of the three unipotent elements $\gamma_1,\gamma_2$ and $(1+\delta)\beta^{-1}$, as claimed.
\end{proof}

	
In \cite[Lemma 4.3]{Pa_Mu_78} it is proved that if $F$ is a field of positive characteristic and $G$ is a finite group, then $FG/J(FG)$ is isomorphic to a product of matrix rings over fields. Now, as a generalization, we show that the following assertion holds.

\begin{lemma}\label{l3.5}
Let $F^\tau G$ be a twisted group algebra of a finite group $G$ over a field $F$ of characteristic $p>0$ such that  the values of the twisting $\tau$  belong to the prime subfield $\mathbb{F}_p$. Then $F^\tau G/J(F^\tau G)$ is isomorphic to a product of matrix rings over fields.
\end{lemma}

\begin{proof}
First note that since Im$(\tau)\subseteq \mathbb{F}_p$, $\mathbb{F}_p^\tau G$ is well defined. Moreover, by \cite[Lemma~6.5.1]{Karpilovsky}, $\mathbb{F}_p ^\tau G\otimes_{\mathbb{F}_p} F\cong F^\tau G$, as $F$-algebras.  Also, since $\mathbb{F}_p$ is a perfect field,  \cite[Proposition~6.5.16]{Karpilovsky} implies that  $J(\mathbb{F}_p ^\tau G)\otimes_{\mathbb{F}_p} F\cong J(F^\tau G)$. Therefore,
$$\frac{F^\tau G}{J(F^\tau G)}\cong \frac{\mathbb{F}_p ^\tau G\otimes_{\mathbb{F}_p} F}{J(\mathbb{F}_p ^\tau G)\otimes_{\mathbb{F}_p} F}\cong
\frac{\mathbb{F}_p ^\tau G}{J(\mathbb{F}_p ^\tau G)}\otimes_{\mathbb{F}_p} F.$$

On the other hand, $\mathbb{F}_p ^\tau G/J(\mathbb{F}_p ^\tau G)\cong \prod_{i=1}^s M_{n_i}(K_i)$, where the $K_i$'s are finite division ring, and hence, commutative fields. Consequently,
$$\frac{F^\tau G}{J(F^\tau G)}\cong  \prod_{i=1}^s (M_{n_i}(K_i) \otimes_{\mathbb{F}_p} F)\cong   \prod_{i=1}^s M_{n_i}(K_i\otimes _{\mathbb{F}_p}F).$$
Since for each $i$, $K_i/{\mathbb{F}_p}$ is a finite separable field extension, it is known that $K_i\otimes _{\mathbb{F}_p}F$ is a finite direct products of some fields (see, for example, \cite[p.~32, Exercise~1]{Bo_Dra_83}). This completes the proof.
\end{proof}

	
Now, with the aid of Lemma~\ref{l3.5} and a similar proof as in Theorem~\ref{twist1} stated above, we can show the validity of the following statement:

\begin{theorem}\label{twist2}
Let $F^\tau G$ be a twisted group algebra of a locally finite group $G$ over a field $F$ of characteristic $p>0$ with $|F|>2$ such that  the values of the twisting $\tau$  belong to the prime subfield $\mathbb{F}_p$.
Then, $\alpha\in (F^\tau G)'$ if and only if $\alpha$ is a product of at most three unipotent elements.
\end{theorem}

	
Now, with Theorem \ref{twist2} at hand, we can easily extract the following generalization of \cite[Theorem~7.4.10]{sehgal}.

\begin{corollary}\label{1}
Let $F^\tau G$ be a twisted group algebra of a locally finite group $G$ over a field $F$ of characteristic $p>0$ with $|F|>2$ such that  the values of the twisting $\tau$  belong to the prime subfield $\mathbb{F}_p$. Then, $F^\tau G$ has no non-zero nilpotent elements if and only if $G$ is an abelian $p'$-group and $\tau$ is symmetric.
\end{corollary}

\begin{proof}
First suppose that $F^\tau G$ has no nilpotent elements. Then, the only unipotent element of $F^\tau G$ is $\bar 1$. Thus $(F^\tau G)'=\bar 1$ employing Theorem \ref{twist2}, i.e., $\mathcal{U}(F^\tau G)$ is abelian. We know that if $\mathcal{G}$ is the group of trivial units $\{a\bar g\mid a\in F^*, g\in G\}$, then $\mathcal{G}/F^*\cong G$. Therefore, $G$ is an abelian group. Furthermore, if $1\neq g$ is of order $p^k$, then
$$\bar g^{p^k}=r_g \cdot \bar 1,~~{\rm ~where~}~~ r_g=\prod_{i=1}^{p^k-1}\tau(g^i,g).$$
Hence, $(\bar g-r_g\cdot\bar1)^{p^k}=0$, showing that $F^\tau G$ has non-zero nilpotent element. Therefore, $G$ is a $p'$-group. Finally, for each $x,y\in G$,
$$\tau (x,y)\overline {xy}=\bar x\bar y=\bar y\bar x=\tau(y,x)\overline{yx}=\tau(y,x)\overline{xy},$$
and so, $\tau(x,y)=\tau (y,x)$, showing that $\tau$ is symmetric.

Conversely, suppose that $G$ is an abelian $p'$-group and $\tau$ is symmetric. We may assume $G$ is finite. Thus $F^\tau G$ is semi-simple by \cite[Theorem~4.2]{Passman}, and commutative (since $G$ is abelian and $\tau$ is symmetric). Now, $F^\tau G$, being a direct product of some fields, has no non-zero nilpotent element, as promised.
\end{proof}

	
\section{Unipotent elements of index $2$ and commutators}

Now we focus on one of the interesting cases: unipotent elements of index $2$. In this case, we need to assume that $FG$ is semi-simple, that is, the Jacobson radical $J(FG)$ is $0$. We begin our work with unipotent matrices of index $2$ over division rings.

\medskip

First, for the reader's convenience and the completeness of the exposition, we recollect the following technicality.

\begin{lemma}\label{comm1}
{\rm (\cite{Pa_Th_61})}
For any field $F$ with at least three elements and any natural number $n$, each element of ${\rm SL}_n(F)$ is a commutator.
\end{lemma}

It is worthwhile noticing that it was showed in \cite{VW} that when $n\geq 3$ and $R$ is a commutative ring satisfying the first Bass stable range condition, then every matrix in ${\rm SL}_n(R)$ is the product of two commutators.

\medskip

We are now in a position to establish the following technical claim.

\begin{lemma}\label{le4.1} Let $D$ be a division ring containing at least three elements and $n$ a positive integer. Then, if $A\in\mathrm{SL}_n(D)$ is a unipotent matrix of index $2$, then $A$ is a commutator in $\mathrm{GL}_n(D)$. Particularly, each unipotent element of $\mathrm{SL}_2(D)$ is a commutator.
\end{lemma}
	
\begin{proof}
Assume that $A\in \mathrm{M}_n(D)$ is a unipotent matrix of index $2$. Then, $(\mathrm I_n-A)^2=0$. If $n=1$, then trivially $A=\mathrm I_n$ is a commutator. Now we consider $n>1$. Let $r$ be the rank of $\mathrm I_n-A$. Then, $r<n$ and by \cite[Lemma~2.1]{Pa_BiDuHaSo}, there exists $P\in \mathrm{GL}_n(D)$  such that $$P(A-\mathrm{I}_n)P^{-1}=\left(\begin{matrix}
		B_1&B_2\\0&0
		\end{matrix}\right),$$ where $B_1 \in \mathrm{M}_{r}(D)$,  $B_2 \in \mathrm{M}_{r\times(n-r)}(D)$, and the rank of matrix  $\left(\begin{matrix}
		B_1&B_2
		\end{matrix}\right)$ is $r$. Moreover, since $(A-\mathrm{I}_n)^2=0$, one has $B_1\left(\begin{matrix}
		B_1&B_2
		\end{matrix}\right)=0$. On the other hand, according to \cite[Lemma 2.2]{Pa_BiDuHaSo}, there exists $T\in\mathrm{GL}_n(D)$ such that $\left(\begin{matrix}
		B_1&B_2
		\end{matrix}\right)T=\left(\begin{matrix}
		\mathrm{I}_r&0
		\end{matrix}\right).$ Hence, $B_1=0$ and the rank of $B_2$ is $r$. By \cite[Lemma 2.2]{Pa_BiDuHaSo} again, there exists $Q\in\mathrm{GL}_{n-r}(D)$ such that $B_2Q=\left\{
             \begin{array}{ll}
               1, & \hbox{if $n=2$;} \\
               \left(\begin{matrix}
		\mathrm{I}_r&0
		\end{matrix}\right), & \hbox{if $n>2$.}
             \end{array}
           \right.$
Therefore, $$\left(\begin{matrix}
		\mathrm{I}_r&0\\0&Q^{-1}
		\end{matrix}\right)P(A-\mathrm{I}_n)P^{-1}\left(\begin{matrix}
		\mathrm{I}_r&0\\0&Q^{-1}
		\end{matrix}\right)^{-1}=\left\{
             \begin{array}{ll}
               \left(\begin{matrix}
		0&1\\0&0
		\end{matrix}\right), & \hbox{if $n=2$;} \\
             \left(\begin{matrix}
		0&\mathrm{I}_r&0\\0&0&0\\0&0&0
		\end{matrix}\right), & \hbox{if $n>2$,}
             \end{array}
           \right.$$
which implies that
$$\left(\begin{matrix}
		\mathrm{I}_r&0\\0&Q^{-1}
		\end{matrix}\right)PAP^{-1}\left(\begin{matrix}
		\mathrm{I}_r&0\\0&Q^{-1}
		\end{matrix}\right)^{-1}=\left\{
             \begin{array}{ll}
               \left(\begin{matrix}
		1&1\\0&1
		\end{matrix}\right), & \hbox{if $n=2$;} \\
             \left(\begin{matrix}
		{\rm I}_r&\mathrm{I}_r&0\\0&{\rm I}_r&0\\0&0&{\rm I}_{n-2r}
		\end{matrix}\right), & \hbox{if $n>2$,}
             \end{array}
           \right.$$
which is the commutator $BCB^{-1}C^{-1}$ in $\mathrm{GL}_n(D)$ by Lemma~\ref{comm1}. Thus, putting $D=\left(\begin{matrix}
		\mathrm{I}_r&0\\0&Q^{-1}
	\end{matrix}\right)P$, we conclude that
$$A=(D^{-1}BD)(D^{-1}CD)(D^{-1}BD)^{-1}(D^{-1}CD)^{-1}$$
is a commutator, as asserted.
\end{proof}

	
The motivation of the following lemma is from \cite[Theorem 1.1]{Pa_Hou_21}, which shows that every matrix in $\mathrm{SL}_n(\mathbb C)$ is a product of at most two commutators of unipotent matrices of index $2$. However, the technique in this paper also works in the case when $\mathbb C$ is replaced by any algebraically closed field $F$. Specifically, we can formulate the following:

\begin{lemma}\label{l4.4}
Let $F$ be an algebraically closed field. Every element in $\mathrm{SL}_n(F)$ is a product of at most two commutators of unipotent matrices of index $2$.
\end{lemma}

	
Now, by the usage of Lemmas \ref{le4.1} and \ref{l4.4}, combined with similar argument used in the previous proofs, we are able to state the next main result of this paper.

\begin{theorem}\label{final} Let $F$ be a field containing at least three elements and $R$ a semi-simple $F$-algebra.
\begin{itemize}
\item[(1)] Every unipotent element of index $2$ in $R$ is a commutator in $R^*$.
\item[(2)] If $F$ is algebraically closed, then every element $\alpha\in R'$ is a product of at most two commutators of unipotent elements of index $2$; thus, $\alpha$ is a product of at most four unipotents of index $2$.
\end{itemize}
\end{theorem}	

In closing our work in this section, in regard to Lemmas~\ref{l3.1} and \ref{le4.1} as well as relevantly to the results obtained in \cite{VW}, we pose the following query.

\begin{problem} Suppose that $D$ is a division ring with at least three elements and $n\geq 1$. If $A\in M_n(D)$ is a product of unipotent matrices of index $2$, is then $A$ a sum of commutators in ${\rm SL}_n(D)$? Precisely, for $n=2$, the question is equivalent to ask whether or not any element of ${\rm SL}_2(D)$ is a sum of commutators.
\end{problem}


\section{An upper bound on the number of unipotents}

In Lemma~\ref{l3.6} we saw that if $F$ is a field, then each element of  $\mathrm{SL}_n(F)$ is a product of at most three unipotents. Now, let $D$ be a non-commutative division ring. Is there a positive integer $\ell$ such that each element of $\mathrm{SL}_n(D)$ can be write as a product of at most $\ell$ unipotents? In this section, we answer this question based on the number $k$ such that each element of $D'$ can be written as a product of at most $k$ commutators. Our basic tool here is the following one:
		
\begin{theorem}\label{bound}
Let $D$ be a non-commutative division ring and $n\geq 2$ a natural number. If each element of $D'$ can be written as a product of at most $k$ commutators, then each $A\in{\rm SL}_n(D)$ is a product of at most $3k+2$ unipotents.
\end{theorem}

\begin{proof}
First, we consider the case when $A\in\mathrm{SL}_n(D)$ is non-central. By application of \cite[Theorem~2.1 and Remark~2.3]{Pa_EgGo_19}, there exists $T\in\mathrm{GL}_n(D)$ such that $TA T^{-1}=XHY$ in which $X$ and $Y$ are, respectively, a lower triangular matrix and an upper triangular matrix whose all diagonal entries are $1$ (hence, $X$ and $Y$ are necessarily unipotents), and $H=\diag (1,\dots,1,s)$, where $1\neq s\in D'$. Suppose that $s$ is a product of $k$ commutators $s_1s_2\cdots s_k$, which all of them are non-trivial, i.e., $s_i\neq 1$. Then, it follows that the matrix $\diag(1,\dots,1,s)$ is a product of $k$ matrices of the form $\diag(1,\dots,1,s_i)$, each of which is a product of three unipotents by virtue of \cite[Lemma~4.2]{Pa_BiDuHaSo_22}. Therefore, $\diag(1,\dots,1,s)$ is a product of $3k$ unipotents and, consequently, $A$ is a product of at most $3k+2$ unipotents.

Now, suppose that $A$ is central. Assume $F$ is the center of $D$. Then, one can write that $A=\lambda\mathrm{I}_n$ for some $\lambda\in F$ such that $\lambda^n\in D'$. If $\lambda^n=1$, then the proof is completed in view of Lemma~\ref{l3.6}. Thus, we may assume that $\lambda^n\neq1$. Then, one can decompose $A$ as
$$A=\mathrm{diag}(\lambda,\lambda,\dots,\lambda,\lambda^{-n+1})\mathrm{diag(1,1,\dots,1,\lambda^n)}.$$
Furthermore, again in virtue of Lemma~\ref{l3.6}, we can infer that the matrix $\mathrm{diag}(\lambda,\lambda,\cdots,\lambda,\lambda^{-n+1})$ is a product of two unipotent elements in $\mathrm{GL}_n(F)\subseteq\mathrm{GL}_n(D)$ as it is non-central. On the other hand, by what we saw above, the matrix $\mathrm{diag}(1,1,\dots,1,\lambda^n)$ is a product of $3k$ unipotents. Therefore, we deduce that $A$ is a product of at most $3k+2$ unipotents, as expected.
\end{proof}


The next comments are worthwhile in order to explain the situation more completely. 
	
\begin{remark}
Let $D$ be a division algebra with center $F$ such that $\dim_F D=m^2$. In \cite{Bo_Dra_80}, Draxl posed the following problem: {\it does there exist an integer $k$ depending on $m$ such that each element of $D'$ is a product of at most $k$ commutators?} This problem seems to be unsolved yet, but in some special cases we know $k$: if $m=2$ (i.e., $D$ is a quaternion algebra), then $k=1$ by \cite[Theorem~1, p.~161]{Bo_Dra_83}, and if $F$ is a non-archimedean local field, then $k=2$ by \cite[Theorem~1]{Bo_Dra_80}. Also, Draxl showed in~\cite{Bo_Dra_80} that this problem has a positive answer for some fields $F$  called ``reasonable", e.g., non-archimedean local fields or global fields. Recently, in \cite{Pa_BiDuHaSo_22}, it showed that if $D$ is tame and totally ramified and if $F$ is Henselian, then the problem also have a positive answer.
\end{remark}

	
\begin{corollary}\label{real2}
Let  $R$ be a finite dimensional semi-simple $F$-algebra such that each Wedderburn component of $R$ is not a non-commutative division ring.

\begin{itemize}
\item[(1)] If $F$ is a real-closed field, then each $\alpha\in R'$  is a product of at most five unipotents.
\item[(2)] If $F$ is a local field, then each $\alpha\in R'$ is a product of at most eight unipotents.
\end{itemize}
\end{corollary}

\begin{proof}
It suffices to assume that $R={\rm M}_n(D)$, where $D$ is a division ring, $\dim_F D<\infty$ and $n\geq 2$.

(1) Assume that $F$ is a real-closed field. By using Frobenius' Theorem in \cite[Theorem]{frob}, either  $D=F$, $D=F(\sqrt{-1})$ or $D=\left(\frac{-1,-1}{F}\right)$ is the quaternion algebra. In the two first cases, the result follows from Lemma~\ref{l3.6}. If $D$ is a quaternion division algebra, consulting with \cite[Theorem~1, p.~161]{Bo_Dra_83}, each element of $D'$ is a commutator. Thus, by Theorem~\ref{bound}, each $\alpha\in R'$ is a product of at most five unipotents.

(2) Now, suppose that $F$ is a local field. If $F$ is archimedean, then either $F\cong\mathbb{R}$ or $F\cong\mathbb{C}$ taking into account \cite[Theorem~12.2.15]{quaternion}. Thus, the result follows from part (1) or Lemma~\ref{l3.6}. Now suppose that  $F$ is non-archimedean and $F\subseteq K$ is the center of $D$. As $K/F$ is a finite extension, \cite[Theorem~12.2.15]{quaternion} implies that $K$ is also a non-archimedean local field. Then, by \cite[Theorem~1]{Bo_Dra_80},  each element of $D'$ is a product of at most $2$ commutators. Consequently, by virtue of Theorem~\ref{bound}, each $\alpha\in R'$ is a product of at most eight unipotents, as promised.
\end{proof}

	
\section{The unipotent radical}

Let $R$ be a ring and $G$ a subgroup of the multiplicative group $R^*$. If each element of $G$ is a unipotent element, then $G$ is called a {\it unipotent group}. An  application of Zorn's Lemma shows that $G$ always has maximal unipotent normal subgroups, but the uniqueness is far from clear. If $G$ has a {\it unique} maximal unipotent normal subgroup, we denote it by $\mathfrak{u}(G)$ and call it the {\it unipotent radical} of $G$. It is known that if $G$ is a linear group, then $\mathfrak{u}(G)$ exists (see, e.g., \cite[Statement~3.1.1]{skew}). In this section, we consider the unipotent radical of two important classes of the general skew linear groups, i.e., the subnormal subgroups and the maximal subgroups, and then extends these results to some twisted group algebras.

\medskip

As the next lemma shows, if $G$ is a subnormal subgroup of a general linear group, then $\mathfrak{u}(G)$ is the trivial group.

\begin{lemma}\label{rad-di}
Let $D$ be a division ring with at least four elements which is locally finite-dimensional over its center and  $n$ a positive integer. If  $G$ is a subnormal subgroup of ${\rm GL}_n(D)$, then $\mathfrak{u}(G)=\langle 1\rangle$.
\end{lemma}

\begin{proof}
The only unipotent element of $D$ is $1$, thus we may assume $n\geq 2$. By \cite[Statement~3.1.1]{skew}, $\mathfrak{u}(G)$ exists. If $\mathfrak{u}(G)$ is non-central, then ${\rm SL}_n(D)=\mathfrak{u}(G)$ by \cite[Lemma~2.3]{gon} and Lemma~\ref{l3.1}. But, clearly, ${\rm SL}_n(D)$  cannot be a unipotent group. Hence, $\mathfrak{u}(G)$ is central. As the only central unipotent element of ${\rm GL}_n(D)$ is ${\rm I}_n$, the proof is completed.
\end{proof}


Using the above lemma, we can prove that any subnormal subgroup of the unit group  of a finite dimensional algebra has the unipotent radical.

\begin{theorem}\label{subnormal}
Let $F$ be a field with at least four elements and $R$ be a finite dimensional $F$-algebra. If $G$ is a subnormal subgroup of $R^*$, then  $\mathfrak{u}(G)=G\cap(1+J(R))$. Particularly, $\mathfrak{u}(R^*)=1+J(R)$.
\end{theorem}

\begin{proof}
Let $R/J(R)\cong\prod_{i=1}^t {\rm M}_{n_i}(D_i)$, where for each $1\leq i\leq t$, $D_i$ is a division ring which is finite dimensional over its center. Now, one verifies that
$$\overline G=\frac{G}{G\cap(1+J(R))}\cong\frac{G(1+J(R))}{1+J(R)}\unlhd\unlhd \frac{R^*}{1+J(R)}\cong \prod_{i=1}^t {\rm GL}_{n_i}(D_i).$$
In view of Lemma~\ref{rad-di}, one detects that $\mathfrak{u}(\overline G)=\langle \overline 1\rangle$.

Putting $U=G\cap(1+J(R))$, $U$ is a unipotent normal subgroup of $G$. If $U\subseteq U_1$ is  another unipotent normal subgroup of $G$, then clearly, $U_1/U$ is a unipotent normal subgroup of $\overline G$. As $\mathfrak{u}(\overline G)=\langle \overline1\rangle$, this implies that $U_1=U$. Therefore, $U$ is a maximal unipotent normal subgroup of $G$.

Suppose now that $V$ is any maximal unipotent normal subgroup of $G$. Then,  by Lemma~\ref{product-unipotent}, $UV$ is unipotent (note that $J(R)$ is  nilpotent). Therefore, $UV$ is a unipotent normal subgroup of $G$ and thus, $UV=U$. This shows that $U=V$. Consequently, $U$ is the unique maximal unipotent normal subgroup of $G$, i.e., $\mathfrak{u}(G)=U$, as needed.
\end{proof}


Now, by the application of Theorem~\ref{subnormal}, we can easily deduce the following result.

\begin{corollary}
Let $F^\tau G$ be a twisted group algebra of a  finite group $G$ over a field $F$ of characteristic $p\geq 0$ with $|F|>3$. If $\mathcal{G}$ is a subnormal subgroup of $(F^\tau G)^*$, then $\mathfrak{u}(\mathcal{G})$ is a $p$-group. Particularly, if $p=0$ or $G$ is a $p'$-group, then $\mathfrak{u}(\mathcal{G})=\langle 1\rangle$.
\end{corollary}


In the next result we show that any maximal subgroup of subnormal subgroups of general linear groups are rarely unipotent. 

\begin{proposition}\label{max-uni}
Let $D$ be a division ring with $|F|>3$, $n$ a positive integer, $N$ a subnormal subgroup of ${\rm GL}_n(D)$ and $M$ a maximal subgroup of $N$. If $M$ is a unipotent group, then $N$ is central and $M=\langle 1\rangle$.
\end{proposition}

\begin{proof}
Firstly, suppose that $n=1$. We know that the only unipotent element of $D$ is $1$, so $M=\langle 1\rangle$. Thus, $N$ is a finite group. Then, by virtue of \cite[Theorem~8]{Herstein}, $N$ is central.

Now, suppose that $n\geq 2$. If $N$ is non-central, then ${\rm SL}_n(D)\subseteq N$ in virtue of \cite[Lemma~2.3]{gon}. On the other hand, Lemma~\ref{l3.1} implies that $M\subseteq {\rm SL}_n(D)$. But, then, $M={\rm SL}_n(D)$, which apparently is not a unipotent group. Therefore, $N$ is central. Thus, the only unipotent element of $N$ is ${\rm I}_n$, and so we deduce $M=\langle 1\rangle$, as required.
\end{proof}


As a valuable consequence, we obtain the following.

\begin{corollary}\label{123}
Let $F$ be an infinite field and let $R$ be a left (right) Artinian $F$-algebra. Then $R^*$ cannot have a unipotent maximal subgroup.
\end{corollary}

\begin{proof}
Let $M$ be a maximal subgroup of $R^*$. Then, it is easy to see that $\overline M=\frac{M(1+ J(R))}{1+ J(R)}$ is either equal to $(R/ J(R))^*\cong \prod_{i=1}^t {\rm GL}_{n_i}(D_i)$, where for each $1\leq i\leq t$, $D_i$ is a division ring, or $\overline M$ is a maximal subgroup of $(R/J(R))^*$. In either cases, utilizing a combination of Lemma~\ref{rad-di} and Proposition~\ref{max-uni}, one deduces that $\overline M$ cannot be a unipotent group and hence so is $M$, as required.
\end{proof}


In order to study the unipotent radical of maximal subgroups we need to some preliminaries. Let $D$ be a division ring with center $F$. Let $G$ be a subgroup of ${\rm GL}_n(D)$. We shall identify the center $F{\rm I}_n$ of ${\rm M}_n(D)$ with $F$. We denote by $F[G]$ the $F$-linear hull of $G$, i.e., the $F$-algebra generated by elements of $G$ over $F$. We also denote by $D^n$ the space of row $n$-vectors over $D$. Then $D^n$ is a $D$-$G$ bi-module in the obvious manner. We say that $G$ is an {\it irreducible} (resp., {\it reducible} or {\it completely reducible}) subgroup of ${\rm GL}_n(D)$ whenever $D^n$ is irreducible (resp., reducible or completely reducible) as $D$-$G$ bi-module. Also, $G$ is called {\it absolutely irreducible} if $F[G] ={\rm M}_n(D)$. Note that by \cite[p.~10]{skew}, the following inclusions are true:
$$\{\text {absolutely irreducible}\}\subseteq \{\text{irreducible}\} \subseteq \{\text {completely reducible}\}.$$

Now, suppose that  $D$ is a division ring of locally finite-dimensional over its center $F$ and $M$ is a maximal subgroup of ${\rm GL}_n(D)$. We would like to determine $\mathfrak{u}(M)$. For $n=1$, clearly, $\mathfrak{u}(M)=\langle 1\rangle$. Let us now $n\geq2$.

If $M$ is reducible,  then by \cite[Corollary~1]{akbari}, it is conjugate to
$$\left\{ \left [\begin{array}{cc}
  A & B \\
  0 & C
\end{array}\right]
\mid A\in {\rm GL}_m(D),C\in {\rm GL}_{n-m}(D),B\in
{\rm M}_{m\times(n-m)}(D)\right\},$$

\medskip

\noindent where $m<n$ is a natural number. Particularly, if $n=2$, then $\mathfrak{u}(M)={\rm U}_2(D)$, the group of unitriangular matrices. But we are left for $n\geq 3$ without knowing the exact answer.

From now on, we assume that $M$ is irreducible. Put $\mathfrak{u}_1=\mathfrak{u}(M)$. As $\mathfrak{u}_1$ is normal in $M$, we have $M\subseteq N_{{\rm GL}_n(D)} (F[\mathfrak{u}_1]^*)$.
Since $M$ is a maximal subgroup of ${\rm GL}_n(D)$, we either have $F[\mathfrak{u}_1]^*\unlhd M$ or $F[\mathfrak{u}_1]^*\unlhd{\rm GL}_n(D)$.

Assume first that $F[\mathfrak{u}_1]^*\unlhd M$. Owing to the classical Clifford's theorem, one concludes that $\mathfrak{u}_1$ is completely reducible (see \cite[p.~4]{skew}). Thus, according to \cite[p.~7]{skew}, $F[\mathfrak{u}_1]$ is semi-simple Artinian. Now, Lemma~\ref{rad-di} assures that $\mathfrak{u}(F[\mathfrak{u}_1]^*)=\langle 1\rangle$. Note that, in this case, $\mathfrak{u}_1\unlhd F[\mathfrak{u}_1]^*$, and hence $\mathfrak{u}_1=\langle 1\rangle$.

Suppose now that $F[\mathfrak{u}_1]^*\unlhd{\rm GL}_n(D)$. If $\mathfrak{u}_1$ is central, then clearly $\mathfrak{u}_1=\langle 1\rangle$; if not, then by the Cartan-Brauer-Hua theorem for matrix ring (see \cite{rosenberg}), we have $F[\mathfrak{u}_1]={\rm M}_n(D)$ (this surely implied that $F[M]={\rm M}_n(D)$, i.e., $M$ is absolutely irreducible). We are left in this case again; but, if $\dim_F D<\infty$ and $D$ is non-commutative, then this case cannot occur:
by \cite[Statement~3.1.1]{skew},  for some $g\in{\rm GL}_n(D)$,   $\mathfrak{u}_1^g$  is a subgroup of ${\rm U}_n(D)$, the group of all unitriangular matrices in ${\rm GL}_n(D)$. Since ${\rm U}_n(D)$ is a nilpotent group \cite[p.~127]{robinson}, so is  $\mathfrak{u}_1$. Then,  we know by \cite[Theorem~2]{Ramezan2013} that $\mathfrak{u}_1$ is abelian, a contradiction.

\medskip

We can now summarize the above arguments in the following assertion.

\begin{theorem}\label{234}
Suppose that $D$ is a division ring which is locally finite-dimensional over its center $F$, with $|F|>2$, $n$ is a positive integer and $M$ is an irreducible maximal subgroup of ${\rm GL}_n(D)$. In either of the following cases $\mathfrak{u}(M)=\langle 1\rangle$:
\begin{itemize}
\item [(1)] if $M$ is not absolutely irreducible;
\item[(2)]  if $\dim_F D<\infty$ and $D$ is non-commutative.
\end{itemize}
\end{theorem}

\end{document}